\documentclass[12pt,reqno]{amsart}
\usepackage{amscd,amssymb,amsmath,mathabx,etex,tikz,tikz-cd,float,cancel}%,array,stmaryrd,graphicx}
\newtheorem{thm}[equation]{Theorem}
\numberwithin{equation}{section}
\newtheorem{cor}[equation]{Corollary}

\newtheorem{rmk}[equation]{Remark}

\newtheorem{lem}[equation]{Lemma}

\newtheorem{defin}[equation]{Definition}

\newtheorem{prop}[equation]{Proposition}

\begin{document}
\raggedbottom \voffset=-.7truein \hoffset=0truein \vsize=8truein
\hsize=6truein \textheight=8truein \textwidth=6truein
\baselineskip=18truept

\def\mapright#1{\ \smash{\mathop{\longrightarrow}\limits^{#1}}\ }
\def\mapleft#1{\smash{\mathop{\longleftarrow}\limits^{#1}}}
\def\mapup#1{\Big\uparrow\rlap{$\vcenter {\hbox {$#1$}}$}}
\def\mapdown#1{\Big\downarrow\rlap{$\vcenter {\hbox {$\ssize{#1}$}}$}}
\def\mapne#1{\nearrow\rlap{$\vcenter {\hbox {$#1$}}$}}
\def\mapse#1{\searrow\rlap{$\vcenter {\hbox {$\ssize{#1}$}}$}}
\def\mapr#1{\smash{\mathop{\rightarrow}\limits^{#1}}}
\def\ss{\smallskip}
\def\ar{\arrow}
\def\vp{v_1^{-1}\pi}
\def\at{{\widetilde\alpha}}
\def\sm{\wedge}
\def\la{\langle}
\def\ra{\rangle}
\def\on{\operatorname}
\def\ol#1{\overline{#1}{}}
\def\spin{\on{Spin}}
\def\cat{\on{cat}}
\def\lbar{\ell}
\def\qed{\quad\rule{8pt}{8pt}\bigskip}
\def\ssize{\scriptstyle}
\def\a{\alpha}
\def\tz{tikzpicture}
\def\bz{{\Bbb Z}}
\def\Rhat{\hat{R}}
\def\im{\on{im}}
\def\ct{\widetilde{C}}
\def\ext{\on{Ext}}
\def\sq{\on{Sq}}
\def\eps{\epsilon}
\def\ar#1{\stackrel {#1}{\rightarrow}}
\def\br{{\bold R}}
\def\bC{{\bold C}}
\def\bA{{\bold A}}
\def\bB{{\bold B}}
\def\bD{{\bold D}}
\def\bh{{\bold H}}
\def\bQ{{\bold Q}}
\def\bP{{\bold P}}
\def\bx{{\bold x}}
\def\bo{{\bold{bo}}}
\def\si{\sigma}
\def\Vbar{{\overline V}}
\def\dbar{{\overline d}}
\def\wbar{{\overline w}}
\def\Sum{\sum}
\def\tfrac{\textstyle\frac}
\def\tb{\textstyle\binom}
\def\Si{\Sigma}
\def\w{\wedge}
\def\equ{\begin{equation}}
\def\AF{\operatorname{AF}}
\def\b{\beta}
\def\G{\Gamma}
\def\D{\Delta}
\def\L{\Lambda}
\def\g{\gamma}
\def\k{\kappa}
\def\psit{\widetilde{\Psi}}
\def\kt{\widetilde{K}}
\def\psiu{{\underline{\Psi}}}
\def\thu{{\underline{\Theta}}}
\def\aee{A_{\text{ee}}}
\def\aeo{A_{\text{eo}}}
\def\aoo{A_{\text{oo}}}
\def\aoe{A_{\text{oe}}}
\def\vbar{{\overline v}}
\def\endeq{\end{equation}}
\def\sn{S^{2n+1}}
\def\zp{\bold Z_p}
\def\cR{{\mathcal R}}
\def\P{{\mathcal P}}
\def\cF{{\mathcal F}}
\def\cQ{{\mathcal Q}}
\def\notint{\cancel\cap}
\def\cj{{\cal J}}
\def\zt{{\bold Z}_2}
\def\bs{{\bold s}}
\def\bof{{\bold f}}
\def\bq{{\bold Q}}
\def\be{{\bold e}}
\def\Hom{\on{Hom}}
\def\ker{\on{ker}}
\def\kot{\widetilde{KO}}
\def\coker{\on{coker}}
\def\da{\downarrow}
\def\colim{\operatornamewithlimits{colim}}
\def\zphat{\bz_2^\wedge}
\def\io{\iota}
\def\Om{\Omega}
\def\Prod{\prod}
\def\e{{\cal E}}
\def\zlt{\Z_{(2)}}
\def\exp{\on{exp}}
\def\abar{{\overline a}}
\def\xbar{{\overline x}}
\def\ybar{{\overline y}}
\def\zbar{{\overline z}}
\def\Rbar{{\overline R}}
\def\nbar{{\overline n}}
\def\Gbar{{\overline G}}
\def\qbar{{\overline q}}
\def\bbar{{\overline b}}
\def\et{{\widetilde E}}
\def\ni{\noindent}
\def\coef{\on{coef}}
\def\den{\on{den}}
\def\lcm{\on{l.c.m.}}
\def\vi{v_1^{-1}}
\def\ot{\otimes}
\def\psibar{{\overline\psi}}
\def\thbar{{\overline\theta}}
\def\mhat{{\hat m}}
\def\ihat{{\hat i}}
\def\jhat{{\hat j}}
\def\khat{{\hat k}}
\def\exc{\on{exc}}
\def\ms{\medskip}
\def\ehat{{\hat e}}
\def\etao{{\eta_{\text{od}}}}
\def\etae{{\eta_{\text{ev}}}}
\def\dirlim{\operatornamewithlimits{dirlim}}
\def\Gt{\widetilde{G}}
\def\lt{\widetilde{\lambda}}
\def\st{\widetilde{s}}
\def\ft{\widetilde{f}}
\def\sgd{\on{sgd}}
\def\lfl{\lfloor}
\def\rfl{\rfloor}
\def\ord{\on{ord}}
\def\gd{{\on{gd}}}
\def\rk{{{\on{rk}}_2}}
\def\nbar{{\overline{n}}}
\def\MC{\on{MC}}
\def\lg{{\on{lg}}}
\def\cB{\mathcal{B}}
\def\cS{\mathcal{S}}
\def\cP{\mathcal{P}}
\def\N{{\Bbb N}}
\def\Z{{\Bbb Z}}
\def\Q{{\Bbb Q}}
\def\R{{\Bbb R}}
\def\C{{\Bbb C}}
\def\l{\left}
\def\r{\right}
\def\mo{\on{mod}}
\def\xt{\times}
\def\notimm{\not\subseteq}
\def\Remark{\noindent{\it  Remark}}
\def\kut{\widetilde{KU}}

\def\*#1{\mathbf{#1}}
\def\0{$\*0$}
\def\1{$\*1$}
\def\22{$(\*2,\*2)$}
\def\33{$(\*3,\*3)$}
\def\ss{\smallskip}
\def\ssum{\sum\limits}
\def\dsum{\displaystyle\sum}
\def\la{\langle}
\def\ra{\rangle}
\def\on{\operatorname}
\def\od{\text{od}}
\def\ev{\text{ev}}
\def\o{\on{o}}
\def\U{\on{U}}
\def\lg{\on{lg}}
\def\a{\alpha}
\def\bz{{\Bbb Z}}
\def\vareps{\varepsilon}
\def\bc{{\bold C}}
\def\bN{{\bold N}}
\def\nut{\widetilde{\nu}}
\def\tfrac{\textstyle\frac}
\def\b{\beta}
\def\G{\Gamma}
\def\g{\gamma}
\def\zt{{\Bbb Z}_2}
\def\pt{\widetilde{p}}
\def\zth{{\bold Z}_2^\wedge}
\def\bs{{\bold s}}
\def\bx{{\bold x}}
\def\bof{{\bold f}}
\def\bq{{\bold Q}}
\def\be{{\bold e}}
\def\lline{\rule{.6in}{.6pt}}
\def\xb{{\overline x}}
\def\xbar{{\overline x}}
\def\ybar{{\overline y}}
\def\zbar{{\overline z}}
\def\ebar{{\overline \be}}
\def\nbar{{\overline n}}
\def\zb{{\overline z}}
\def\Mbar{{\overline M}}
\def\et{{\widetilde e}}
\def\ni{\noindent}
\def\ms{\medskip}
\def\ehat{{\hat e}}
\def\what{{\widehat w}}
\def\Yhat{{\widehat Y}}
\def\Wbar{{\overline W}}
\def\minp{\min\nolimits'}
\def\mul{\on{mul}}
\def\N{{\Bbb N}}
\def\Z{{\Bbb Z}}
\def\Q{{\Bbb Q}}
\def\R{{\Bbb R}}
\def\C{{\Bbb C}}
\def\notint{\cancel\cap}
\def\se{\operatorname{secat}}
\def\cS{\mathcal S}
\def\cR{\mathcal R}
\def\el{\ell}
\def\TC{\on{TC}}
\def\dstyle{\displaystyle}
\def\ds{\dstyle}
\def\mt{\widetilde{\mu}}
\def\zcl{\on{zcl}}
\def\Vb#1{{\overline{V_{#1}}}}

\def\Remark{\noindent{\it  Remark}}
\title
{$K$-theory and immersions of spatial polygon spaces}
\author{Donald M. Davis}
\address{Department of Mathematics, Lehigh University\\Bethlehem, PA 18015, USA}
\email{dmd1@lehigh.edu}
\date{January 13, 2019}


\maketitle
\begin{abstract} For $\ell$ an $n$-tuple of positive numbers, $N(\ell)$ denotes the space of isometry classes of oriented $n$-gons in $\R^3$ with side lengths specified by $\ell$. If $\ell$ is generic, $N(\ell)$ is a $2(n-3)$-manifold. We determine the algebra $K(N(\ell))$ and use this to obtain nonimmersions of $N(\ell)$ in Euclidean space for several families of $\ell$. We also use obstruction theory to tell exactly when $N(\ell)$ immerses in $\R^{4n-14}$ for two families of $\ell$'s.
 \end{abstract}
 \section{Introduction}\label{introsec}
If $\ell=(\ell_1,\ldots,\ell_n)$ is an $n$-tuple of positive real numbers, $N(\ell)$ denotes the space of oriented $n$-gons in $\R^3$ with consecutive sides of the specified lengths, identified under translation and rotation of $\R^3$. Edges of the polygon are allowed to intersect. A more formal definition is
$$N(\ell)=\{(z_1,\ldots,z_n)\in (S^2)^n: \sum_{i=1}^n \ell_iz_i=0\}/SO(3).$$
See \cite{HK}, \cite{HR}, or \cite{Kl}. It is clear from this definition that the diffeomorphism type of $N(\ell)$ is not affected by changing the order of the $\ell_i$'s.

Let $[\![n]\!]=\{1,\ldots,n\}$. If there is no subset $S\subset[\![n]\!]$ for which $\ds\sum_{i\in S}\ell_i=\sum_{i\not\in S}\ell_i$, $\ell$ is said to be {\it generic}. If $\ell$ is generic, then $N(\ell)$ is a $2(n-3)$-manifold. (\cite[p.285]{HK}) Throughout this paper, $\ell$ is always assumed to be generic. Generic spatial polygon spaces are classified in \cite{HR} by their genetic code, which is a collection of subsets of $[\![n]\!]$ determined by $\ell$, which we will define at the beginning of Section \ref{sec2}. The manifolds $N(\ell)$ and $N(\ell')$ are diffeomorphic if and only if they have the same genetic code. All genetic codes for $n\le9$ are listed in \cite{web}. For example, there are 134 diffeomorphism classes of nonempty spatial 7-gon spaces.

Hausmann and Knutson determined the integral cohomology ring $H^*(N(\ell))$ in \cite{HK}. We give our interpretation of this result, in terms of the genetic code, in Theorem \ref{cohthm}. Since $H^*(N(\ell))$ is generated by elements of $H^2(N(\ell))$ and is torsion-free, the Chern character effectively determines the complex $K$-theory algebra $K(N(\ell))$. However, a general statement of this seems too complicated to be useful. In Theorems \ref{ring}, \ref{Knk1}, and \ref{few}, we will give explicit results for the ring $K(N(\ell))$ when the genetic code of $\ell$ is $\{\{n,k\}\}$ or $\{\{n,k,1\}\}$, and in general modulo a certain ideal.

Our main goal is to obtain nonimmersion results for spatial polygon spaces in Euclidean space. We use an old result, which we state later as Theorem \ref{Bak}, which tells how $K$-theoretic Chern classes can yield lower bounds for the geometric dimension of stable vector bundles.
This result is applied to the stable normal bundle of $N(\ell)$, using a result of \cite{HK} about the cohomology Chern classes of the tangent bundle of $N(\ell)$.
We will obtain the following three nonimmersion results. Throughout the paper we let $m=n-3$, so that $N(\ell)$ is a $2m$-manifold. Also, $\a(m)$ denotes the number of 1's in the binary expansion of $m$.
\begin{thm}\label{nonimm1} If the genetic code of $\ell$ is $\{\{n,k\}\}$ with $k<n$, the $2m$-manifold $N(\ell)$ cannot be immersed in $\R^{4m-2\a(m)-1}$.\end{thm}
This is the same as the standard result that can be proved for $C P^m$ using complex $K$-theory. For $\a(m)\le 8$, these results for $C P^m$ are very close to optimal. (See \cite{DCP}.)
Note that $CP^m$ is a spatial polygon space with genetic code $\{\{n\}\}$.(\cite[Expl 2.6]{H})

\begin{thm}\label{nonimm2} If the genetic code of $\ell$ is $\{\{n,k,1\}\}$ with $1<k<n$, the $2m$-manifold $N(\ell)$ cannot be immersed in $\R^{4m-2\a(m)-3}$. If $k$ is odd, it cannot be immersed in $\R^{4m-2\a(m)-1}$.\end{thm}

The sets in the genetic code are called {\it genes}. In the following result, as always, $n$ is the length of $\ell$, and $m=n-3$.
\begin{thm}\label{nonimm3} Let $s+1$ be the size of the largest gene of $\ell$. Let $M=\max(i-\nu\binom{m+i}i:i\le m-s)$. Then the $2m$-manifold $N(\ell)$ cannot be immersed in $\R^{2m+2M-1}$.\end{thm}
\noindent Here and throughout, $\nu\binom nk$ denotes the exponent of 2 in the binomial coefficient.
In Table \ref{T1}, we will tabulate some values for Theorem \ref{nonimm3}.

Existence of immersions is usually proved by obstruction theory. Recall that since $N(\ell)$ is a $2m$-manifold, it certainly immerses in $\R^{4m-1}$. In Section \ref{sec5}, we prove the following immersion result.
\begin{thm}\label{immthm} Let $n\ge5$. If the genetic code of $\ell$ is $\{\{n,k\}\}$, then $N(\ell)$ can be immersed in $\R^{4m-2}$ if and only if it is not the case that $m$ is a $2$-power and $k$ is even. If the genetic code of $\ell$ is $\{\{n,k,1\}\}$, then $N(\ell)$ can always be immersed in $\R^{4m-2}$.\end{thm}

\begin{rmk} {\bf Optimal immersions}: {\rm Combining Theorems \ref{nonimm1}, \ref{nonimm2}, and \ref{immthm}, we find that when $n-3$ is a 2-power, immersions of the $2m$-manifold $N_{n,k}$ in $\R^{4m-1}$ when $k$ is even, and of both $N_{n,k}$ and $N_{n,k,1}$ in $\R^{4m-2}$ when $k$ is odd are optimal.}\end{rmk}

We thank Martin Bendersky and Jean-Claude Hausmann for helpful suggestions.

\section{General results and proof of Theorem \ref{nonimm1}}\label{sec2}
In this section, we provide the background for all our proofs, and apply them to prove Theorem \ref{nonimm1}.

We assume without loss of generality that $\ell_1\le\cdots\le\ell_n$.
 For a length vector $\ell$, a subset  $S\subseteq[\![n]\!]$ is called {\it short} if $\ds \sum_{i\in S}\ell_i<\ds\sum_{i\not\in S}\ell_i$.  A partial order is defined on  the power set of $[\![n]\!]$, based on inclusions of sets and sizes of numbers.(\cite[\S4]{HR}) The {\it genetic code} of $\ell$ is defined to be the set of maximal elements under this ordering in the set of all short subsets containing $n$. For example, the genetic code of $\ell=(\underbrace{1,\ldots,1}_k,\underbrace{2,\ldots,2}_{n-k-1},2n-k-5)$ is $\{\{n,k\}\}$, and the genetic code of $\ell=(\frac12,\underbrace{1,\ldots,1}_{k-1},\underbrace{2,\ldots,2}_{n-k-1},2n-k-6)$ is $\{\{n,k,1\}\}$.

In \cite{D3}, we defined a {\it gee} to be a gene without listing the $n$, and an element $\le$ a gee is called a {\it subgee}. So the subgees are just the subsets $T$ of $[\![n-1]\!]$ such that $T\cup\{n\}$ is short. Our interpretation in \cite{D3} of the theorem of \cite{HK} is as follows.
\begin{thm}\label{cohthm} The integral cohomology ring $H^*(N(\ell))$ is generated by degree-$2$ elements $R$ and $V_i$, $1\le i<n$, with relations
\begin{itemize}
\item[a.] For a subset $S\subseteq[\![n-1]\!]$, let $V_S=\ds\prod_{i\in S}V_i$. Then $V_S=0$ if $S$ is not a subgee.
    \item[b.] For all $i$, $RV_i+V_i^2=0$.
\item[c.] For each subgee $T$ with $|T|\ge n-2-d$, there is a relation $\mathcal R_T$ in $H^{2d}(N(\ell))$:
$$ \sum_{S\notint T} R^{d-|S|}V_S=0.$$
\end{itemize}
\end{thm}

The following useful result is quickly deduced from \cite{HK}.
\begin{prop} The total Chern class of the tangent bundle $\tau(N(\ell))$ satisfies
$$(1+R)c(\tau(N(\ell)))=\prod_{i=1}^{n-1}(1+2V_i+R).$$\label{HKprop}
\end{prop}
\begin{proof} Hausmann and Knudson utilize a space $UP$ called the upper path space, and in \cite[Remark 7.5c]{HK} note that
$$c(\tau(UP))=(1+R)\prod_{i=1}^{n-1}(1+V_i+R)(1+V_i) \text{ and } (1+R)^2c(\tau(N(\ell)))=c(\tau(UP)).$$
The result is immediate from this and the relation $V_i^2+RV_i=0$.\end{proof}
\noindent Note that $V_i=0$ unless $\{i\}$ is a subgee of $\ell$, which usually allows for cancellation of a factor $(1+R)$.

\begin{cor} There are complex line bundles $L_R$ and $L_i$, $1\le i\le n-1$, over $N(\ell)$ such that $c_1(L_R)=R$, $c_1(L_i)=V_i$, and the complex bundles $L_R\oplus\tau(N(\ell))$ and $\ds\bigoplus_{i=1}^{n-1}L_i^2L_R$ are stably isomorphic.\label{tan}\end{cor}
\begin{proof} Since $c_1$ defines a bijection of complex line bundles over $X$ with $H^2(X;\Z)$ satisfying $c_1(L_1L_2)=c_1(L_1)+c_1(L_2)$, the two complex bundles in the corollary exist and have the same total Chern class by Proposition \ref{HKprop}, and hence have the same Chern character. Since $ch:K(X)\otimes\Q\to H^*(X;\Q)$ is an isomorphism, the stable bundles are equal in $K(X)$ mod torsion. Since $H^*(N(\ell))$ is confined to even dimensions, $K(N(\ell))$ has no torsion, and so the bundles are stably isomorphic.\end{proof}
\noindent If $V_i=0$, then $L_i^2L_R=L_R$.

Now we apply these results to $N_{n,k}$,  which is defined to be $N(\ell)$ when the genetic code of $\ell$ is $\{\{n,k\}\}$. First we have the result for integral cohomology.
\begin{prop}\label{cohnk} There is a ring isomorphism

\begin{eqnarray}\label{coh}H^*(N_{n,k})&=&\Z[R,V_1,\ldots,V_k]/(V_iV_j=0 \text{ if }i\ne j,\ V_i^2=-RV_i,\\
&&\nonumber  V_1^m=\cdots=V_k^m, R^m=(-1)^m(k-1)V_i^m, R^{m+1}=0=V_i^{m+1}),\end{eqnarray}
with $|R|=|V_i|=2$.
Bases for the nonzero groups $\widetilde{H}^{2j}(N_{n,k})$ are given by $\{R^j,V_1^j,\ldots,V_k^j\}$, $1\le j<m$,  and $\{V_1^m\}$.\end{prop}
\begin{proof}
 The only nonempty subgees are $\{i\}$ with $1\le i\le k$. Relations of type c only occur in grading $m$. So, using relation b, a basis in grading less than $m$ consists of just the $j$th powers. In grading $m$, there is a relation $\mathcal R_i$ for each $i$ from 1 to $k$ of the form
$$R^m+\sum_{j\ne i}R^{m-1}V_j=0,$$
or equivalently
$$R^m+(-1)^{m-1}\sum_{j\ne i}V_j^{m}=0.$$
This system clearly reduces to the claim in (\ref{coh}).\end{proof}

Next we deduce the additive structure of $K(N_{n,k})$.
\begin{defin}\label{def} For any $\ell$, let $L_R$ and $L_i$ be the complex line bundles over $N(\ell)$ with $c_1(L_R)=R$ and $c_1(L_i)=V_i$. Let $\a_i=[L_i-1]$ and $\b=[L_R-1]\in\kt(N(\ell))$.\end{defin}
\begin{prop}  The abelian group $\kt(N_{n,k})$ is  free  with basis $$\{\a_1^j:\ 1\le j\le m\}\cup\{\a_2^j,\ldots,\a_k^j,\b^j:1\le j<m\}.$$\label{ind}\end{prop}
\begin{proof} We have $ch(\a_i^j)\equiv V_i^j$ and $ch(\b^j)\equiv R^j$,  where $\equiv$ means mod terms of higher degree.
So the specified $\a_i^j$ and $\b^j$ are elements in $\kt(N_{n,k})$ on which the first component of $ch$ gives an isomorphism to a basis for $H^*(N_{n,k})$. Standard methods imply that for a space with only
even-dimensional cohomology and no torsion, this implies the claim.

Indeed, let $K_{2j}(N_{n,k})$ denote the quotient of $K(N_{n,k})$ by elements trivial on the $(2j-1)$-skeleton. A result in \cite{Ad}, along with collapsing of the Atiyah-Hirzebruch spectral sequence, implies that
$$ch_j:K_{2j}(N_{n,k})/K_{2j+2}(N_{n,k})\to H^{2j}(K_{n,k})$$ is an isomorphism. Now by downward induction the split short exact sequence
$$0\to K_{2j+2}(N_{n,k})\to K_{2j}(N_{n,k})\to H^{2j}(N_{n,k})\to 0$$
implies that each $K_{2j}(N_{n,k})$ is a free abelian group generated by powers $\ge j$.

\end{proof}

\begin{thm}\label{ring} The multiplicative relations in $K(N_{n,k})$ are $\a_i\a_j=0$ for $i\ne j$, $\a_i\b=-\a_i^2/(1+\a_i)$, $\a_1^m=\cdots=\a_k^m$, $\b^m=(-1)^m(k-1)\a_i^m$, $\b^{m+1}=0=\a_i^{m+1}$.\end{thm}
\begin{proof} Since $ch$ is an isomorphism  $K(X)\ot\Q\to H^*(X;\Q)$ and we have seen that bases correspond, it suffices to show that $ch$ sends asserted relations to 0. Using the cohomology relations stated in (\ref{coh}), all are clear except the one for $\a_i\b$. Letting $V=V_i$ and noting that when multiplied by $V$, $R$ acts like $-V$, we have
$$ch(\a_i\b)=(e^V-1)(e^R-1)=(e^V-1)(e^{-V}-1)=-(e^V-1)\tfrac{e^V-1}{e^V}=-\tfrac{(e^V-1)^2}{1+(e^V-1)}=ch(\tfrac{-a_i^2}{1+\a_i}).$$
\end{proof}

\begin{lem} For any $i$, $L_i^2L_R-1=2\a_i+\b+\a_i\b\in\kt(N(\ell))$.\label{lemrel}\end{lem}
\begin{proof} Let $V=V_i$.
Since $R$ acts as $-V$ when multiplied by $V$, we have in $H^*(N(\ell))$
\begin{equation}\label{nice}(2V+R)^j=R^j+((2V-V)^j-(-V)^j)=\begin{cases}2V^j+R^j&j\text{ odd}\\ R^j&j\text{ even.}\end{cases}\end{equation}
 We obtain
\begin{eqnarray*}ch(L_i^2L_R-1)&=&\sum_{j\ge1}\tfrac{(2V+R)^j}{j!}\\
&=&2(V+\tfrac{V^3}{3!}+\tfrac{V^5}{5!}+\cdots)+e^R-1\\
&=&(e^V-1)-(e^{-V}-1)+e^R-1\\
&=&(e^V-1)+(e^V-1)(1+e^{-V}-1)+e^R-1\\
&=&2(e^V-1)+(e^V-1)(e^{-V}-1)+(e^R-1)\\
&=&ch(2\a_i+\a_i\b+\b).\end{eqnarray*}
Hence we have equality in $K(N(\ell))$.
\end{proof}

To obtain nonimmersions, we use
the following result, which we quote from \cite[Theorem 4.4]{Baker}, although earlier versions apparently exist.
\begin{thm} \label{Bak}For complex bundles $\theta$ over finite complexes, there are natural classes $\G(\theta)\in K(X)\ot\Z[\frac12]$ satisfying $\G(\theta_1+\theta_2)=\G(\theta_1)\G(\theta_2)$ and $\G(L)=1+\frac{L-1}2$  if $L$ is a (complex) line bundle.
If $\theta$ is a stably complex vector bundle over a finite complex $X$ which, as a real bundle, is stably  isomorphic to a bundle of (real) dimension $2s+1$, then $2^s\G(\theta)\in K(X)$.\end{thm}
\noindent The components of $\G(X)$ are just divided versions of the $K$-theoretic Chern classes.

Now we can prove Theorem \ref{nonimm1}.

\begin{proof}[Proof of Theorem \ref{nonimm1}] Let $\eta$ be the normal bundle of an immersion of $N_{n,k}$ in $\R^t$ for some $t$.
By Corollary \ref{tan}, the tangent bundle of $N_{n,k}$ is stably isomorphic to $\ds\bigoplus_{i=1}^kL_i^2L_R\oplus(n-2-k)L_R$, and hence, with $\G(-)$ as in Theorem \ref{Bak},
\begin{eqnarray}
\G(\eta)&=&\G(\tau)^{-1}=\prod_{i=1}^k\G(L_i^2L_R)^{-1}\cdot\G(L_R)^{-(m+1-k)}\nonumber\\
&=&\prod_{i=1}^k(1+\a_i+\tfrac12\b+\tfrac12\a_i\b)^{-1}\cdot(1+\tfrac12\b)^{-(m+1-k)}\nonumber\\
&=&\prod_{i=1}^k(1+\a_i)^{-1}\cdot(1+\tfrac12\b)^{-(m+1)}\label{gnor}\\
&=&(1+\tfrac12\b)^{-(m+1)}+\sum_{i,j}c_{i,j}\a_i^j,\nonumber\end{eqnarray}
using the relations in Theorem \ref{ring} in the last step. Here $c_{i,j}$ is an element of $\Z_{(2)}$ whose value does not matter.
These terms are independent by Proposition \ref{ind}, and
one of the terms is $\binom{-m-1}{m-1}\b^{m-1}/2^{m-1}$. The binomial coefficient here is $(-1)^{m-1}\binom{2m-1}{m-1}=(-1)^{m-1}\frac12\binom{2m}m$, so its exponent of 2 is $\a(m)-1$. Thus $2^{m-\a(m)-1}\cdot\binom{-m-1}{m-1}\b^{m-1}/2^{m-1}$ is not integral, and so by Theorem \ref{Bak}, $\eta$ is not stably equivalent to a bundle of dimension $2(m-\a(m)-1)+1$. If a $d$-dimensional manifold immerses in $\R^{d+c}$, then its normal bundle is $c$-dimensional. Thus we obtain the theorem.\end{proof}

Exactly as we did for planar polygon spaces in \cite{mf}, we can prove that the total Stiefel-Whitney class of the tangent bundle of $N(\ell)$ is $(1+R)^{m+1}\in H^*(N(\ell);\zt)$ with $|R|=2$, where $m$, as throughout this paper, is 3 less than the length of $\ell$.
Similarly to \cite[Corollary 1.5]{mf}, for each $m$ from 16 to 31, Stiefel-Whitney classes give a nonimmersion of $N_{n,k}$ in $\R^{61}$. For these values of $m$, Theorem \ref{nonimm1} gives nonimmersions of $N_{n,k}$ in Euclidean spaces of the following dimensions: 61, 63, 67, 69, 75, 77, 81, 83, 91, 93, 97, 99, 105, 107, 111, and 113.

\section{Proof of Theorem \ref{nonimm3}}
We can expand these results easily if we restrict to cases in which relations of type c do not appear in $H^*(N(\ell))$. Let $s$ denote the maximal size of a gee of $\ell$.
In grading $\le 2(m-s)$, there are no type-c relations. In this range, a basis for $H^*(N(\ell))$ consists of all $R^iV_S$ such that $S$ is a subgee and $i+|S|\le m-s$. This includes the empty subgee.

Using the notation and methods used earlier, we have the following result.
\begin{thm} \label{few}Let $s$ equal the maximal size of a gee of $\ell$, and $k$ the largest size-$1$ subgee of $\ell$. The ring $K(N(\ell))$ is generated by $\a_1,\ldots,\a_k$, and $\b$. The quotient, modulo $(m-s+1)$-fold (or more) products of these, is
$$\Z[\a_1,\ldots,\a_k,\b]/(\a_i\b+\tfrac{\a_i^2}{1+\a_i}, \prod_{i\in T}\a_i\text{: $T$ not a subgee}).$$
A basis for this quotient consists of all $\b^i\prod_{j\in S}\a_j$ such that $S$ is a subgee and $i+|S|\le m-s$.\end{thm}
\begin{proof} Since $ch(\a_i^j)\equiv V_i^j$ and $ch(\b^j)\equiv R^j$ mod higher-degree classes, the Atiyah-Hirzebruch argument in the proof of Proposition \ref{ind} implies that these stated classes generate. The relations are exactly those which give a cohomology relation when $ch$ is applied. The first type of relation is obtained as in the proof of Theorem \ref{ring}, and the second type follows from type-a relations in cohomology. For every type-c relation in cohomology, there is a unique $K$-theory relation mapping to it, and this relation will involve only $(m-s+1)$-fold (or more) products of the generators. Thus they will not affect the quotient of $K(N(\ell))$ being considered.\end{proof}

\begin{proof}[Proof of Theorem \ref{nonimm3}]
Corollary \ref{tan}, Lemma \ref{lemrel}, and (\ref{gnor}) all apply exactly as in the proof of Theorem \ref{nonimm1}, as does the argument about independence of pure $\b$ classes from those involving any $\a_i$'s. The only difference from the situation of the previous section is that all we can assert is that certain multiples of $\b^i$ are nonzero for $i\le m-s$.
Thus $\G(\eta)$ contains independent terms $2^{-i}\binom{-m-1}i\b^i$ for $i\le m-s$. Note that $\binom{-m-1}i=\pm\binom{m+i}i$.
If $M$ is as in Theorem \ref{nonimm3},
then $\G(\eta)$ contains an independent term $2^{-M}\b^i$, so $2^{M-1}\G(\eta)$ is not integral. Therefore by Theorem \ref{Bak}, the dimension of the normal bundle is greater than $2(M-1)+1$, and so the manifold does not immerse in this codimension.\end{proof}

The results implied by Stiefel-Whitney classes are exactly those in which an $i$ which determines $M$ has $\binom{m+i}i$ odd. In Table \ref{T1}, we list the nonimmersion dimension implied by Theorem \ref{nonimm3} for $2m$-dimensional $N(\ell)$ having largest gee size $s$. Those having value $\le 61$ (except for the 55 in columns 7 and 8) are also implied by Stiefel-Whitney classes, but 61 is the largest that Stiefel-Whitney classes can imply in any of the tabulated cases.

\begin{table}[H]
\caption{Nonimmersion dimensions implied by Theorem \ref{nonimm3}}
\label{T1}
\begin{tabular}{c|cccccccc}
$m/s$&$1$&$2$&$3$&$4$&$5$&$6$&$7$&$8$\\
\hline
$16$&$61$&$59$&$57$&$55$&$53$&$51$&$49$&$47$\\
$17$&$63$&$61$&$61$&$57$&$57$&$53$&$53$&$49$\\
$18$&$67$&$65$&$61$&$61$&$61$&$59$&$55$&$53$\\
$19$&$69$&$67$&$67$&$61$&$61$&$61$&$61$&$55$\\
$20$&$75$&$73$&$71$&$69$&$63$&$61$&$61$&$61$\\
$21$&$77$&$75$&$75$&$71$&$71$&$63$&$63$&$61$\\
$22$&$81$&$79$&$75$&$75$&$75$&$73$&$65$&$63$\\
$23$&$83$&$81$&$81$&$75$&$75$&$75$&$75$&$65$\\
$24$&$91$&$89$&$87$&$85$&$83$&$81$&$79$&$77$\\
$25$&$93$&$91$&$91$&$87$&$87$&$83$&$83$&$79$\\
$26$&$97$&$95$&$91$&$91$&$91$&$89$&$85$&$83$\\
$27$&$99$&$97$&$97$&$91$&$91$&$91$&$91$&$85$\\
$28$&$105$&$103$&$101$&$99$&$95$&$93$&$91$&$89$\\
$29$&$107$&$105$&$105$&$101$&$101$&$95$&$95$&$91$\\
$30$&$111$&$109$&$105$&$105$&$105$&$103$&$97$&$95$\\
$31$&$113$&$111$&$111$&$105$&$105$&$105$&$105$&$97$
\end{tabular}
\end{table}

\section{Proof of Theorem \ref{nonimm2}} \label{sec3}
For $1<k<n$, let $N_{n,k,1}=N(\ell)$ when the genetic code of $\ell$ is $\{\{n,k,1\}\}$.
In this section we determine the algebra $K(N_{n,k,1})$ and use it to prove the nonimmersion Theorem \ref{nonimm2}.  Throughout we let $m=n-3$. We begin with (integral) cohomology.
\begin{thm} \label{nk1coh} Let $N=N_{n,k,1}$. The algebra $H^*(N)$ is generated by degree-$2$ classes $R,V_1,\ldots,V_k$. In grading $2i$ with $2\le i\le m-2$, the only relations are due to $RV_i=-V_i^2$, and $V_iV_j=0$ if $2\le i<j$, so a basis is given by $$\{R^i, V_1^i,\ldots,V_k^i, V_1V_2^{i-1},\ldots, V_1V_k^{i-1}\}.$$ A basis for $H^{2(m-1)}(N)$ is $\{V_1^{m-1},V_k^{m-1},V_1V_2^{m-2},\ldots,V_1V_k^{m-2}\}$, with $V_2^{m-1}=\cdots=V_k^{m-1}$ and $R^{m-1}=
(-1)^{m-1}(k-2)V_k^{m-1}$. In $H^{2m}(N)\approx\Z$, we have $R^m=V_2^m=\cdots=V_k^m=0$, $V_1V_2^{m-1}=\cdots=V_1V_k^{m-1}$ a generator, and $V_1^m=(k-2)V_1V_k^{m-1}$. Finally $H^i(N)=0$ if $i>2m$.\end{thm}
\begin{proof} This follows easily from Theorem \ref{cohthm}. The type-c relations in $H^{2(m-1)}(N)$ are $R^{m-1}+(-1)^{m-2}\ds\sum_{j\ge2,j\ne i}V_j^{m-1}=0$. One can solve this system by row-reduction, but it is easier just to verify that the relations stated in the theorem are consistent with these. Similarly, in $H^{2m}(N)$, one can check that the relations stated in the theorem are consistent with the type-c relations $\mathcal{R}_{1,i}$, $\mathcal{R}_1$, and $\mathcal{R}_i$, $i>1$:
$$R^m+(-1)^{m-1}\sum_{j\ge 2,j\ne i}V_j^m=0,\qquad R^m+(-1)^{m-1}\sum_{j=2}^kV_j^m=0,$$
and $$R^m+(-1)^{m-1}\sum_{j\ne i}V_j^m+(-1)^{m-2}\sum_{j\ge2,j\ne i}V_1V_j^{m-1}=0.$$
\end{proof}

\begin{thm} \label{Knk1} The algebra $K(N_{n,k,1})$ is generated by classes $\a_1,\ldots,\a_k,\b$ with relations $\a_i\b=-\a_i^2/(1+\a_i)$, $\a_i\a_j=0$ if $2\le i<j$, $\a_2^{m-1}=\cdots=\a_k^{m-1}$, and $\b^{m-1}=(-1)^{m-1}(k-2)\a_k^{m-1}$. A basis consists of
\begin{eqnarray*}&&\{1,\a_i^j\ (1\le i\le k,\ 1\le j\le m-2), \b,\ldots,\b^{m-2}, \a_1^{m-1},\a_k^{m-1},\\
&&\a_1\a_i^{j} (2\le i\le k,\ 1\le j\le m-2),\ \a_1\a_k^{m-1}\}.\end{eqnarray*}\end{thm}
\begin{proof} The generators are as defined previously, and the relations are easily established by applying $ch$ and using the cohomology relations. The stated $K$-theory relations imply the following additional relations: $\a_1^t\a_i^j=\a_1\a_i^{j+t-1}$, $\a_i^m=0$ ($2\le i\le k$), $\b^m=0$, and $\a_1^m=(k-2)\a_1\a_k^{m-1}$. These relations enable reduction of everything to terms in our asserted basis, and the listed monomials are linearly independent since applying $ch$ to them yields linearly independent cohomology classes.\end{proof}

\begin{rmk} {\rm Although in the cases considered here the $K$-theory relations have corresponded exactly to the cohomology relations (except for the $\a_i\b$ relation), this will not continue. We have studied the case in which the genetic code is $\{\{n,k,2\}\}$. In this case some $K$-theory relations contain an additional term. The complicated details cause us not to include it in this paper.}\end{rmk}

\begin{proof}[Proof of Theorem \ref{nonimm2}] The first part follows from Theorem \ref{nonimm3}, using $i=m-2$ if $m$ is even, and $i=m-3$ if $m$ is odd. Standard techniques imply that $\nu\binom{2m-2}{m-2}=\a(m)-1$ if $m$ is even, and $\nu\binom{2m-3}{m-3}=\a(m)-2$ if $m$ is odd. In both cases $i-\nu\binom{m+i}i=m-\a(m)-1$.

We might hope to use $2^{-(m-1)}\binom{2m-1}{m-1}\b^{m-1}=2^{\a(m)-m}u\b^{m-1}$ with $u$ odd to show that $2^{m-\a(m)-1}\G(\eta)$ is not integral and deduce a nonimmersion in $\R^{4m-2\a(m)-1}$. If $k$ is even, the relation $\b^{m-1}=(-1)^{m-1}(k-2)\a_k^{m-1}$ prevents this from working. We will show that it works when $k$ is odd, but since $\b^{m-1}$ is not independent of $\a_k^{m-1}$, we must consider other terms.  We have, similarly to (\ref{gnor}),
\begin{eqnarray*} \G(\eta)&=&\prod_{i=1}^k((1+\a_i)(1+\tfrac12\b))^{-1}\cdot(1+\tfrac12\b)^{-(m-k+1)}\\
&=&(1+\a_1)^{-1}(1+\a_2+\cdots+\a_k)^{-1}(1+\tfrac12\b)^{-(m+1)}\\
&=&(1+\a_2+\cdots+\a_k)^{-1}\bigl((1+\tfrac12\b)^{-(m+1)}-\tfrac{\a_1}{1+\a_1}(1-\tfrac{\a_1}{2(1+\a_1)})^{-(m+1)}\bigr)\\
&=&\bigl(1+\sum_{t>0,j\ge2}\a_j^t\bigr)\bigl((1+\tfrac12\b)^{-(m+1)}-\a_1(1+\a_1)^m(1+\tfrac12\a_1)^{-(m+1)}\bigr).\end{eqnarray*}
The desired term, $2^{-(m-1)}\binom{-m-1}{m-1}\b^{m-1}$, is apparent, but we must consider possible cancelling terms.

The terms $\a_1^{t}$ and $\a_1\a_j^{t}$ are independent of $\b^{m-1}$ and so need not be considered, but we must consider the coefficients of $\a_j^{m-1}$ for $2\le j\le k$, since these terms are related to $\b^{m-1}$. Since (as already used) $\b$ acts as $-\a_j/(1+\a_j)$ when multiplied by $\a_j$, we obtain as possible cancelling terms
\begin{equation}\label{all}\sum_{t>0,j\ge2}\a_j^t\bigl(1-\tfrac{\a_j}{2(1+\a_j)}\bigr)^{-(m+1)}=\ds\sum_{t>0,j\ge2}\a_j^t(1+\a_j)^{m+1}/(1+\tfrac12\a_j)^{m+1}.\end{equation}
We will prove in Proposition \ref{combprop} that the 2-exponent of the coefficient of $x^i$ in $(1+x)^{m+1}/(1+\frac12x)^{m+1}$ is $\ge \a(m)-m$ for $i\le m-1$. Thus the coefficient of $\a_j^{m-1}$ in the RHS of (\ref{all}) has 2-exponent $\ge\a(m)-m$. Since all these $\a_j^{m-1}$ are equal and have the same coefficient, and there is an even number of them, this RHS contains the term $2^v\a_k^{m-1}$ with $v>\a(m)-m$. Since $\b^{m-1}$ is an odd multiple of $\a_k^{m-1}$, the combination of all terms is an odd multiple of $2^{\a(m)-m}$ times the basis element $\a_k^{m-1}$, and so $2^{m-\a(m)-1}\G(\eta)$ is not integral, from which the result follows.
\end{proof}
\begin{prop}\label{combprop} For $i\le m-1$, $\nu([x^i]:((1+2x)/(1+x))^{m+1})\ge i+\a(m)-m$.\end{prop}
\begin{proof} The desired coefficient equals $\ds\sum_j2^j\tbinom{m+1}j\tbinom{-(m+1)}{i-j}$. A formula of Gould (\cite[(1.71)]{Gou} is $\ds\sum2^j\tbinom xj\tbinom y{i-j}=\sum\tbinom xj\tbinom{x+y-j}{i-j}$, and so our coefficient equals
\begin{eqnarray*}&&\sum_j\tbinom{m+1}j\tbinom{-j}{i-j}=\pm\sum(-1)^j\tbinom{m+1}j\tbinom{i-1}{i-j}\\
&=&\pm([x^i]:(1-x)^{m+1}(1+x)^{i-1})=\pm([x^i]:(1-x^2)^{i-1}(1-x)^{m+2-i}).\end{eqnarray*}
The evaluation of this coefficient  varies slightly with the parity of $m$ and $i$. We consider here the case $m=2r$ and $i=2r-2t$, with $t>0$. Other parities lead to very slight modifications.

Up to sign, the coefficient is
$$\sum_{k=r-2t-1}^{r-t}(-1)^k\tbinom{2r-2t-1}k\tbinom{2t+2}{2r-2t-2k}
=\tbinom{2r}r\sum_{k=r-2t-1}^{r-t}\tfrac{(r\cdots(k+1))\cdot(r\cdots(2r-2t-k))}{(2r)\cdots(2r-2t)}\cdot c_k,$$
where $c_k=(-1)^k\binom{2t+2}{2r-2t-2k}$ is an integer whose value is not relevant. For each value of $k$, either $r-t\ge k+1$ or $r-t\ge 2r-2t-k$. Thus one of the two products in the numerator of the fraction contains the product $r\cdots (r-t)$. We use this to cancel all but the 2 in each of the even factors in the denominator, leaving $2^{t+1}$ times an odd number in the denominator. Thus every term $T$ in the sum has $\nu(T)\ge -(t+1)$. Since $\nu\binom{2r}r=\a(r)=\a(m)$, we obtain that the 2-exponent in our expression is $\ge\a(m)-t-1\ge\a(m)-2t=\a(m)-(m-i)$, as claimed.
\end{proof}
Our proof shows that Proposition \ref{combprop} can be strengthened quite a bit, but we settle for what we need.
 \def\line{\rule{.6in}{.6pt}}

 \section{Obstruction theory: proof of Theorem \ref{immthm}}\label{sec5}
 We adapt \cite[Theorem 5.1]{LP} to our situation as follows.
 \begin{thm}\label{LPthm} The $2m$-manifold $N(\ell)$ can be immersed in $\R^{4m-2}$ if and only if
 \begin{itemize}
 \item[a.] $m$ is even and there exists $y\in H^{2m-2}(N(\ell);\zt)$ such that
 $$i_*(\sq^2y+w_2(\eta)y)=\rho_4(c_m(\eta))\in H^{2m}(N(\ell);\Z_4),\text{ or}$$
 \item[b.] $m$ is odd and $\rho_2(c_m(\eta))=0\in H^{2m}(N(\ell);\zt)$.\end{itemize}\end{thm}
 Here $c_m(\eta)\in H^{2m}(N(\ell);\Z)$ is the Chern class of the stable normal bundle $\eta$, and $i_*$, $\rho_4$, and $\rho_2$ are induced by coefficient homomorphisms $\zt\hookrightarrow\Z_4$, $\Z\twoheadrightarrow\Z_4$, and $\Z\twoheadrightarrow\Z_2$.

 \begin{proof}[Proof of Theorem \ref{immthm}]
 By Proposition \ref{HKprop}, the total Chern class of the normal bundle $\eta$ for any $N(\ell)$ is
 \begin{equation}\label{chern}\prod_{i=1}^k(1+2V_i+R)^{-1}\cdot(1+R)^{-(m+1-k)}.\end{equation}
 Using (\ref{nice}), we have
 \begin{eqnarray*}(1+2V_i+R)^{-1}&=&\sum_{j\ge0}(-1)^j(2V_i+R)^j\\
 &=&\sum_{j\ge0}(-1)^jR^j+2\sum_{j\text{ odd}}(-1)^jV_i^j\\
 &=&(1+R)^{-1}-2V_i(1-V_i^2)^{-1}.\end{eqnarray*}

 Now we specialize to $N_{n,k}$ and $N_{n,k,1}$.
 Since $V_iV_j=0$ for $i\ne j$ in $H^*(N_{n,k})$, and $R$ acts as $-V_i$ when multiplied by $V_i$, we obtain in $H^*(N_{n,k})$
 \begin{eqnarray*}\prod_{i=1}^k(1+2V_i+R)^{-1}&=&(1+R)^{-k}-2(1+R)^{-(k-1)}\sum_{i=1}^kV_i(1-V_i^2)^{-1}\\
 &=&(1+R)^{-k}-2\sum_{i=1}^kV_i(1-V_i^2)^{-1}(1-V_i)^{-(k-1)}.\end{eqnarray*}
 In $H^*(N_{n,k,1})$, there are additional terms since $V_1V_j\ne0$, but they will be divisible by 4.
 Thus, from (\ref{chern}), in both $N_{n,k}$ and $N_{n,k,1}$,
 \begin{equation}\label{cm}c_m(\eta)\equiv[(1+R)^{-(m+1)}-2\sum_{i=1}^kV_i(1-V_i^2)^{-1}(1-V_i)^{-m}]_{2m}\mod 4,\end{equation}
 where $[-]_{2m}$ denotes the component in grading $2m$. (Recall $|R|=|V_i|=2$.)
 If $m$ is odd, the coefficient of $R^m$ in $(1+R)^{-(m+1)}$ is even, and so Theorem \ref{LPthm}(b) implies the immersion.

Now we restrict to even values of $m$.
 Mod 2, we have, with $|V|=2$,
 $$[V(1-V^2)^{-1}(1-V)^{-m}]_{2m}\equiv V[(1+V)^{-(m+2)}]_{2(m-1)}=\tbinom{-(m+2)}{m-1}V^m\equiv\tbinom{2m}{m-1}V^m.$$
This coefficient is even when $m$ is even, and so the second part of (\ref{cm}) is 0 mod 4.

 Since $\nu\binom{-(m+1)}m=\nu\binom{2m}m=\a(m)$, we find that, mod 4,  $c_m(\eta)$ equals $2R^m$ if $m$ is a 2-power, and is 0 otherwise, so we obtain the immersion when $m$ is not a 2-power.
 By Theorem \ref{nk1coh}, $R^m=0$ in $H^{2m}(N_{n,k,1})$ and so we obtain the immersion in this case.
 By Proposition \ref{cohnk}, $2R^m\ne0\in H^{2m}(N_{n,k};\Z_4)$  iff $k$ is even, and so we obtain the immersion when $m$ is a 2-power and $k$ is odd.
 Now we evaluate the indeterminacy, $\sq^2y+w_2(\eta)y$, when $m$ is a 2-power and $k$ is even.

 Note that $H^{2m}(N_{n,k};\zt)\approx\zt$, so we just say whether terms are 0. We have $\sq^2(V_i^{m-1})\ne0$ since $m$ is even, and $\sq^2(R^{m-1})\ne0$ since $m$ is even and $k$ is even.
 Also, $w_2(\eta)=(m+1)R$, so $w_2(\eta)V_i^{m-1}\ne0$ since $m$ is even, and $w_2(\eta)R^{m-1}\ne0$ since $m$ is even and $k$ is even. So the indeterminacy is 0 in all cases, and we obtain the nonimmersion for $N_{n,k}$ when $m$ is a 2-power and $k$ is even.
 \end{proof}

\end{document}